\documentclass[11pt,a4paper]{amsart}

\usepackage[T1]{fontenc}
\usepackage[utf8]{inputenc}
\usepackage{lmodern}
\usepackage{amsmath,amssymb,amsthm,mathtools,amsaddr}
\usepackage{upgreek}
\usepackage{xcolor}
\usepackage{geometry}
\usepackage{microtype}
\usepackage{enumerate}
\usepackage{hyperref}
\usepackage[nameinlink,noabbrev]{cleveref}

\hypersetup{
  colorlinks=true,
  linkcolor=violet!85!black,
  citecolor=violet!80!black,
  urlcolor=violet!90!black,
  pdftitle={A~Sharp Surface-Area Extension of Vaaler's Theorem},
  pdfauthor={Micha\l{} Zwierzy\'nski},
  pdfkeywords={sections of the~cube, Vaaler's theorem, surface area, skeletal measure, ridge, orthoschemes, Gaussian solid angle}
}
\newtheorem{theorem}{Theorem}[section]
\newtheorem{proposition}[theorem]{Proposition}
\newtheorem{lemma}[theorem]{Lemma}
\newtheorem{corollary}[theorem]{Corollary}
\newtheorem{conjecture}[theorem]{Conjecture}
\theoremstyle{definition}

\newtheorem{question}[theorem]{Question}
\theoremstyle{remark}
\newtheorem{remark}[theorem]{Remark}

\newcommand{\R}{\mathbb{R}}
\newcommand{\B}{\mathbb{B}}
\newcommand{\cH}{\mathcal{H}}
\newcommand{\vol}{\operatorname{vol}}
\newcommand{\dist}{\operatorname{dist}}
\newcommand{\aff}{\operatorname{aff}}
\newcommand{\lin}{\operatorname{lin}}
\newcommand{\conv}{\operatorname{conv}}
\newcommand{\relint}{\operatorname{relint}}
\newcommand{\dd}{\,\mathrm{d}}

\title[A~Sharp Surface-Area Extension of Vaaler's Theorem]{A~Sharp Surface-Area Extension of Vaaler's Theorem}
\author{Micha\l{} Zwierzy\'nski}

\address{Warsaw University of Technology\\
Faculty of Mathematics and Information Science\\
ul. Koszykowa 75\\
00-662 Warsaw, Poland}

\email{Michal.Zwierzynski@pw.edu.pl}
\email{ORCID: 0000-0002-9627-1563}



\begin{document}

\begin{abstract}
Karasev proved, in dimensions two and three, a~sharp surface-area counterpart of a~polyhedral extension of Vaaler's theorem \cite[Theorem~1.2]{Karasev}.  We remove the~dimension restriction.  More precisely, if an~$n$-dimensional convex polytope $P$ contains the~origin in its interior and the~affine hull of every nonempty proper face of codimension $k\in\{1,\ldots,n\}$ is at distance at least $\sqrt{k}$ from the~origin, then
$$
  \cH^{n-1}(\partial P)\geqslant n2^n.
$$
Consequently, the~boundary of every $n$-dimensional linear section of the~cube $[-1,1]^N$ has $(n-1)$-dimensional measure at least $n2^n$.  This confirms, in all dimensions, a~conjecture of Grigory M. Ivanov recorded by Karasev \cite[Section~1]{Karasev}.  The~proof combines the~Rogers--Karasev flag decomposition with a~dimension-free Gaussian comparison for orthoschemes.  Its main step is an~ordered-square substitution which converts all face-distance assumptions into a~pointwise domination of a~single positive integral.  We also disprove the~naive extension to all skeletal measures, formulate a~ridge-skeleton conjecture, and establish two partial results toward it.
\end{abstract}

\maketitle

\tableofcontents

\section{Introduction}

\noindent Let $[-1,1]^N$ be the~$N$-dimensional cube and let $L\subset\R^N$ be an~$n$-dimensional linear subspace.  Vaaler's theorem states that
$$
  \vol_n\bigl([-1,1]^N\cap L\bigr)\geqslant 2^n
$$
(see \cite{Vaaler}).  An~exposition that also includes Vaaler's more general product-of-balls formulation is given in \cite[Theorem~1.1 and Corollary~1.1]{Zong}. A~sharp combinatorial analogue was obtained by Ball and Prodromou \cite[Theorem~1]{BallProdromou}: every such section supports a~probability measure whose second-moment operator dominates the~identity on $L$. Karasev recently observed \cite[Section~2]{Karasev} that an~argument of Rogers \cite{Rogers}, originally developed in connection with sphere packings, gives a~short proof of Vaaler's theorem and of the~following more general polyhedral statement \cite[Theorem~1.1]{Karasev}: if $0\in\operatorname{int}P$ and every codimension-$k$ face $F$ of $P$ satisfies
\begin{equation}\label{eq:face-distance}
  \dist(0,\aff F)\geqslant\sqrt{k},
\end{equation}
then $\vol_n(P)\geqslant 2^n$.

Karasev also considered the~corresponding lower bound for surface area.  As recorded in \cite[Section~1]{Karasev}, the~assertion for sections of the~cube was conjectured by Grigory M. Ivanov in a~private communication with Roman Karasev.  Karasev proved the~stronger polyhedral version for $n=2,3$ \cite[Theorem~1.2]{Karasev}.  His final step reduces to monotonicity of a~quotient involving a~spherical simplex \cite[Section~3]{Karasev}. It is obvious on the~circle and is handled there by a~special spherical-triangle calculation in dimension three.  The~purpose of this note is to remove the~restriction on the~dimension.

Our main result is the~following theorem.

\begin{theorem}\label{thm:main}
Let $n\geqslant1$, and let $P\subset\R^n$ be an~$n$-dimensional convex polytope with $0\in\operatorname{int}P$.  Suppose that every nonempty proper face $F$ of codimension $k\in\{1,\ldots,n\}$ satisfies \eqref{eq:face-distance}.  Then
\begin{equation}\label{eq:main}
  \cH^{n-1}(\partial P)\geqslant n2^n.
\end{equation}
The~constant is sharp, as is seen by taking $P=[-1,1]^n$.
\end{theorem}

Throughout, a~polytope means a~compact convex polyhedron.  All volumes and Hausdorff measures are computed in the~relevant Euclidean affine hull.

The~consequence for cube sections is immediate from the~face geometry of the~cube -- a~detailed deduction is given in Section \ref{sec:global}.

\begin{corollary}\label{cor:cube}
For integers $1\leqslant n\leqslant N$ and every $n$-dimensional linear subspace $L\subset\R^N$,
\begin{equation}\label{eq:cube-section}
  \cH^{n-1}\!\left(\partial_L\bigl([-1,1]^N\cap L\bigr)\right)
  \geqslant n2^n,
\end{equation}
where $\partial_L$ denotes the~boundary relative to $L$.
\end{corollary}

We next describe the~proof in some detail.  This also indicates where the~dimension-free ingredient enters.

\begin{enumerate}[(i)]
  \item For every maximal flag of faces of $P$, choose the~points of the~faces nearest to the~origin.  Their convex hull is a~\emph{flag simplex}.  These simplices cover $P$, while their facets opposite the~origin cover $\partial P$.
  \item Replace the~nearest point in each face by the~orthogonal projection of the~origin onto the~affine hull of that face.  The~resulting simplex is an~orthoscheme.  The~Rogers--Karasev contraction argument \cite[Section~2]{Karasev} shows that this replacement can only decrease the~fraction of the~simplex lying in the~unit ball.  Since the~original and the~orthoscheme have the~same altitude over their opposite facets, the~required local surface-to-sector ratio can only decrease under the~replacement.
  \item In orthogonal coordinates, the~vertices of the~orthoscheme are
  $$
    b_k=(\beta_1,\ldots,\beta_k,0,\ldots,0),
  $$
  and \eqref{eq:face-distance} becomes the~family of prefix inequalities $\beta_1^2+\cdots+\beta_k^2\geqslant k$.  The~volume of the~unit-ball sector divided by the~volume of the~unit ball is its Gaussian solid
  angle.  Expressing this angle as an~ordered Gaussian integral, the~dimension-free comparison lemma shows that the~prefix inequalities make the~relevant surface-to-sector ratio no smaller than for the~standard cube orthoscheme.
  \item The~resulting local bound is
  $$
    \frac{\text{area of the~opposite facet}}
         {\text{volume of the~unit-ball sector}}
    \geqslant \frac{n2^n}{\vol_n(\B^n)}.
  $$
  Summation over all flags finishes the~argument because the~opposite facets triangulate $\partial P$ and the~ball sectors partition $\B^n$.
\end{enumerate}

This comparison uses all prefix inequalities simultaneously.  In particular, it avoids a~dimension-by-dimension analysis of spherical simplices.

\section{From flag simplices to orthoschemes}

\noindent Write $\B^n=\{x\in\R^n:|x|\leqslant1\}$ and $\kappa_n=\vol_n(\B^n)$.  An~ordered $n$-simplex $B=\conv\{b_0,\ldots,b_n\}$ is called an~\emph{orthoscheme} if its successive edge vectors
$$
  b_1-b_0,\ b_2-b_1,\ \ldots,\ b_n-b_{n-1}
$$
are pairwise orthogonal.  Equivalently, for every $k=1,\ldots,n-1$, the~direction spaces of $\conv\{b_0,\ldots,b_k\}$ and $\conv\{b_k,\ldots,b_n\}$ are orthogonal.  This is the~terminology used in \cite[Definition~2.1]{Karasev}.  We first dispose of a~harmless degeneracy that otherwise obscures some Jacobian calculations.  For a~face $G$, the~notation $\lin G$ means the~direction space $\lin(G-G)$ of $\aff G$, while $\relint G$ denotes the~interior of $G$ relative to $\aff G$.

\begin{lemma}[Generic-center reduction]\label{lem:generic}
It is enough to prove Theorem \ref{thm:main} for polytopes with the~following property: whenever $F\subset G$ are incident faces with $\dim G=\dim F+1$, the~orthogonal projections of the~origin onto $\aff G$ and $\aff F$ are distinct.
\end{lemma}

\begin{proof}
Fix a~polytope $P$ satisfying the~assumptions of Theorem \ref{thm:main}.  For an~incident pair $F\subset G$ as in the~statement, the~set of points $v$ for which the~projections of $v$ onto $\aff G$ and $\aff F$ coincide is
$$
  E_{G,F}=\aff F+(\lin G)^\perp.
$$
It is an~affine hyperplane in $\R^n$.  There are only finitely many such pairs.  Consequently, for every sufficiently small $\delta\in(0,1)$ we may choose $v\in\operatorname{int}P$ with $|v|<\delta$ outside their union.

Set
$$
  P_{v,\delta}=\frac{P-v}{1-\delta}.
$$
If $F$ has codimension $k$, then the~corresponding face $F_{v,\delta}$ of $P_{v,\delta}$ satisfies
\begin{align*}
  \dist(0,\aff F_{v,\delta})
  =\frac{\dist(v,\aff F)}{1-\delta}
  \geqslant\frac{\dist(0,\aff F)-|v|}{1-\delta}
   \geqslant\frac{\sqrt{k}-\delta}{1-\delta}
   \geqslant\sqrt{k}.
\end{align*}
Thus $P_{v,\delta}$ satisfies the~hypotheses and has the~desired genericity property.  If Theorem \ref{thm:main} has been proved under this genericity assumption, then
$$
  \frac{\cH^{n-1}(\partial P)}{(1-\delta)^{n-1}}
  =\cH^{n-1}(\partial P_{v,\delta})\geqslant n2^n.
$$
Letting $\delta\to 0^+$, proves the~assertion for $P$, which completes the proof.
\end{proof}

We now assume the~genericity in Lemma \ref{lem:generic}.  For every face $F\subseteq P$, let $a_F$ be a~point of $F$ nearest to the~origin.  In particular, $a_P=0$.  Given a~maximal flag
\begin{equation}\label{eq:flag}
  P=F_0\supset F_1\supset\cdots\supset F_n,
  \qquad \operatorname{codim}F_k=k,
\end{equation}
put $a_k=a_{F_k}$ and $A_{\mathcal F}=\conv\{a_0,a_1,\ldots,a_n\}$.
These are the~flag simplices.  This is the~flag subdivision used by Karasev \cite[Section~2]{Karasev}, following Rogers \cite{Rogers}: on each face $G$, join $a_G$ to the~already subdivided facets of $G$.  Indeed, for $x\in G\setminus\{a_G\}$, define
$$
  \tau_G(x)=\max\{t\geqslant0:a_G+t(x-a_G)\in G\}.
$$
Then $\tau_G(x)\geqslant1$, and the~exit point $a_G+\tau_G(x)(x-a_G)$ belongs to the~relative boundary of $G$, hence to a~facet of $G$. The~facet is unique away from a~lower-dimensional set.  Induction on $\dim G$ shows that the~resulting cells cover $G$ and have disjoint relative interiors up to measure zero.  If $a_G\notin\relint G$, the~cells based on facets containing $a_G$ are degenerate and may be discarded.  Applied to $P$ and to its facets, this gives
\begin{equation}\label{eq:flag-partitions}
  P=\bigcup_{\mathcal F}A_{\mathcal F},
  \qquad
  \partial P=\bigcup_{\mathcal F}
  \conv\{a_1,\ldots,a_n\}
\end{equation}
in the~measure-theoretic sense needed below.

For a~fixed flag \eqref{eq:flag}, let $b_k$ be the~orthogonal projection of $0$ onto $\aff F_k$ and set
$$
  B_{\mathcal F}=\conv\{b_0,b_1,\ldots,b_n\}.
$$
Notice that $b_0=0$.  The~next lemma makes the~first contraction step in Karasev's proof \cite[Section~2]{Karasev} fully explicit.

\begin{lemma}[Rogers--Karasev contraction]\label{lem:contraction}
Let $A_{\mathcal F}$ be nondegenerate.  Then $B_{\mathcal F}$ is a~nondegenerate orthoscheme, and the~affine map
\begin{align}
\label{eq:TFmapDefinition}
  T_{\mathcal F}:A_{\mathcal F}\rightarrow B_{\mathcal F},
  \qquad T_{\mathcal F}(a_k)=b_k,
\end{align}
satisfies $|T_{\mathcal F}x|\leqslant|x|$ for every $x\in A_{\mathcal F}$.
\end{lemma}

\begin{proof}
Since $\aff F_k\subset\aff F_{k-1}$, the~point $b_k$ is also the~orthogonal projection of $b_{k-1}$ onto $\aff F_k$.  Hence
$$
  b_k-b_{k-1}\perp\lin F_k.
$$
For $j>k$, the~vector $b_j-b_{j-1}$ belongs to $\lin F_{j-1}\subseteq\lin F_k$.  The~successive edge vectors of $B_{\mathcal F}$ are therefore pairwise orthogonal.  They are nonzero by the~genericity assumption, so $B_{\mathcal F}$ is a~nondegenerate orthoscheme.

For $x\in A_{\mathcal F}$, let $x=\sum_{i=0}^n\lambda_i a_i$ be its barycentric representation, where $\lambda_i\geqslant0$ and $\sum_i\lambda_i=1$, and define
$$
  x^{(k)}=\sum_{i=0}^k\lambda_i b_i+
          \sum_{i=k+1}^n\lambda_i a_i,
  \qquad k=0,\ldots,n.
$$
Thus $x^{(0)}=x$ and $x^{(n)}=T_{\mathcal F}x$.  We prove $|x^{(k)}|\leqslant|x^{(k-1)}|$.

At the~$k$th step put $u_k=a_k-b_k$ and, if $u_k\neq0$, consider the~hyperplane
$$
  H_k=\{y:(y-b_k)\mathbin{\cdot}u_k=0\}.
$$
Both $u_k\in\lin F_k$ and $b_k\perp\lin F_k$, so $H_k$ contains the~origin.  It also contains $b_0,\ldots,b_k$: every successive edge $b_j-b_{j-1}$ with $j\leqslant k$ is perpendicular to $\lin F_k$.

Moreover, $a_k$ is the~point of the~convex set $F_k$ nearest to $b_k$: for $y\in\aff F_k$ one has
$$
  |y|^2=|b_k|^2+|y-b_k|^2.
$$
For $y\in F_k$ and $0\leqslant t\leqslant1$, convexity gives $a_k+t(y-a_k)\in F_k$.  Since $t=0$ minimizes the~squared distance of this segment from $b_k$, its right derivative is nonnegative:
$$
  0\leqslant\left.\frac{\mathrm d}{\mathrm dt}\right|_{t\to 0^+}
  |a_k+t(y-a_k)-b_k|^2
  =2(y-a_k)\mathbin{\cdot}u_k.
$$
Consequently,
$$
  (y-a_k)\mathbin{\cdot}u_k\geqslant0,
  \qquad
  (y-b_k)\mathbin{\cdot}u_k\geqslant|u_k|^2.
$$
In particular, $F_k$ lies on the~$a_k$-side of $H_k$.  Since $a_i\in F_i\subseteq F_k$ for $i\geqslant k$, both $x^{(k-1)}$ and $x^{(k)}$ lie in that closed halfspace.  Their difference is
$$
  x^{(k)}-x^{(k-1)}=-\lambda_k u_k.
$$
Thus the~component parallel to $H_k$ is unchanged, while the~nonnegative normal component decreases by $\lambda_k|u_k|$ and cannot change sign.  As $0\in H_k$, this proves $|x^{(k)}|\leqslant|x^{(k-1)}|$.  If $u_k=0$, the~two points are equal.  Iteration over $k=1,\ldots,n$ proves the~claim without any assumption on the~intermediate simplices.
\end{proof}

The~contraction gives precisely the~local reduction needed for surface area.

\begin{proposition}\label{prop:local-reduction}
For every nondegenerate flag simplex $A=A_{\mathcal F}$ and its orthoscheme $B=B_{\mathcal F}$,
\begin{equation}\label{eq:local-reduction}
  \frac{\cH^{n-1}(\conv\{a_1,\ldots,a_n\})}
       {\vol_n(A\cap\B^n)}
  \geqslant
  \frac{\cH^{n-1}(\conv\{b_1,\ldots,b_n\})}
       {\vol_n(B\cap\B^n)}.
\end{equation}
\end{proposition}

\begin{proof}
By Lemma \ref{lem:contraction},
$$
  T_{\mathcal F}(A\cap\B^n)\subseteq B\cap\B^n.
$$
The~map $T_{\mathcal F}$ has constant Jacobian $\vol_n(B)/\vol_n(A)$ (see \eqref{eq:TFmapDefinition}), and hence
\begin{equation}\label{eq:relative-volume}
  \frac{\vol_n(A\cap\B^n)}{\vol_n(A)}
  \leqslant
  \frac{\vol_n(B\cap\B^n)}{\vol_n(B)}.
\end{equation}
The~two simplices have the~same altitude over the~displayed opposite facets: both affine hulls are $\aff F_1$, whose distance from $0$ is $|b_1|$.  Consequently,
$$
  \frac{\cH^{n-1}(\conv\{a_1,\ldots,a_n\})}{\vol_n(A)}
  =
  \frac{\cH^{n-1}(\conv\{b_1,\ldots,b_n\})}{\vol_n(B)}.
$$
Dividing this identity by \eqref{eq:relative-volume} gives \eqref{eq:local-reduction}.
\end{proof}

\section{A~Gaussian comparison for orthoschemes}

\noindent The~essential dimension-free estimate is the~following integral lemma. Ordered Gaussian integrals of this type arise in formulas for solid angles of faces of type-$B$ Weyl chambers (see, for instance,~\cite[Equation (3.5)]{GodlandKabluchko}).  The~point of the~following lemma is the~comparison under the~prefix conditions $q_1+\cdots+q_k\geq k$.

\begin{lemma}[Ordered-square comparison]\label{lem:ordered-square}
Let $q_1,\ldots,q_n>0$ and put $r_k=q_1+\cdots+q_k$ for $k=1,\ldots,n$.
Assume that $r_k\geqslant k$ for $k=1,\ldots,n$.  For
$$
  D_n=\{z\in\R^n:z_1\geqslant z_2\geqslant\cdots\geqslant z_n\geqslant0\}
$$
define
$$
  \mathcal I(q_1,\ldots,q_n)
  =\int_{D_n}\exp\!\left(-\frac12\sum_{j=1}^nq_jz_j^2\right)\dd z.
$$
Then
\begin{equation}\label{eq:ordered-square}
  \sqrt{q_1}\,\mathcal I(q_1,\ldots,q_n)\leqslant \mathcal I(1,\ldots,1).
\end{equation}
\end{lemma}

\begin{proof}
For $n=1$, a~direct rescaling gives
$\sqrt{q_1}I(q_1)=\mathcal I(1)$.  Suppose that $n\geqslant2$.
Since $q_1=r_1\geqslant1$, define a~map from $D_n$ onto itself by
$$
  w_1^2-w_2^2=q_1(z_1^2-z_2^2),
  \qquad
  w_j=z_j\quad (j=2,\ldots,n).
$$
Equivalently, $w_1=\sqrt{z_2^2+q_1(z_1^2-z_2^2)}$.
This map is a~bijection on the~interiors of the~two cones.  Its
Jacobian satisfies
$$
  \dd z=\frac{w_1}{q_1z_1}\dd w.
$$
Moreover, $w_1^2=q_1z_1^2-(q_1-1)z_2^2\leqslant q_1z_1^2$, and hence
\begin{align}
  \label{eq:w1q1z1Ineq}
  \frac{w_1}{\sqrt{q_1}z_1}\leqslant1.
\end{align}

The~quadratic form transforms as
\[
  \sum_{j=1}^nq_jz_j^2
  =
  w_1^2+(q_1+q_2-1)w_2^2+\sum_{j=3}^nq_jw_j^2.
\]
Set $\widetilde q_1=1$, $\widetilde q_2=q_1+q_2-1$, and $\widetilde q_j=q_j$ for $j\geqslant3$,
and let $\widetilde r_k=\widetilde q_1+\cdots+\widetilde q_k$.
Then
$$
  \widetilde r_1=1,
  \qquad
  \widetilde r_k=r_k\geqslant k\quad (k\geqslant2).
$$
Writing $w_{n+1}=0$ and summing by parts, we obtain
\begin{align}
  \label{eq:qjwjQuadraticIneq}
  \sum_{j=1}^n\widetilde q_jw_j^2
  =
  \sum_{k=1}^n\widetilde r_k(w_k^2-w_{k+1}^2)
  \geqslant
  \sum_{k=1}^nk(w_k^2-w_{k+1}^2)
  =
  \sum_{j=1}^nw_j^2.
\end{align}
Consequently, by \eqref{eq:w1q1z1Ineq} and \eqref{eq:qjwjQuadraticIneq}, we get
\begin{align*}
  \sqrt{q_1}\mathcal I(q_1,\ldots,q_n)
  &=
  \int_{D_n}
  \exp\!\left(-\frac12\sum_{j=1}^n\widetilde q_jw_j^2\right)
  \frac{w_1}{\sqrt{q_1}z_1}\dd w\\
  &\leqslant
  \int_{D_n}
  \exp\!\left(-\frac12\sum_{j=1}^nw_j^2\right)\dd w
  =\mathcal I(1,\ldots,1).
\end{align*}
\end{proof}

Now we apply Lemma~\ref{lem:ordered-square} to a~ball sector cut out by an~orthoscheme.

\begin{proposition}[Orthoscheme sector estimate]\label{prop:orthoscheme}
Let
$$
  B=\conv\{0,b_1,\ldots,b_n\}\subset\R^n
$$
be a~nondegenerate orthoscheme such that $|b_k|\geqslant\sqrt{k}$ for every $k$.  Then
\begin{equation}\label{eq:orthoscheme-sector}
  \frac{\cH^{n-1}(\conv\{b_1,\ldots,b_n\})}
       {\vol_n(B\cap\B^n)}
  \geqslant\frac{n2^n}{\kappa_n}.
\end{equation}
\end{proposition}

\begin{proof}
Choose orthonormal coordinates along the~successive mutually orthogonal edges of $B$.  With suitable choices of signs,
\begin{equation}\label{eq:orthoscheme-coordinates}
  b_k=(\beta_1,\ldots,\beta_k,0,\ldots,0),
  \qquad \beta_j>0.
\end{equation}
Set $q_j=\beta_j^2$.  The~assumptions are precisely
\begin{equation}\label{eq:prefix-beta}
  r_k:=q_1+\cdots+q_k=|b_k|^2\geqslant k.
\end{equation}

The~cone of $B$ at the~origin is
$$
  C_{\beta}
  =\{(\beta_1z_1,\ldots,\beta_nz_n):
       z_1\geqslant\cdots\geqslant z_n\geqslant0\}.
$$
The~simplex itself is obtained by adding the~inequality $z_1\leqslant1$.  Since $\beta_1\geqslant1$, every $x\in C_{\beta}\cap\B^n$ satisfies $\beta_1z_1\leqslant|x|\leqslant1$, and therefore $z_1\leqslant1$.  Thus
\begin{equation}\label{eq:sector-equality}
  B\cap\B^n=C_{\beta}\cap\B^n.
\end{equation}

The~base of $B$ lies in the~hyperplane $x_1=\beta_1$.  Since $\vol_n(B)=\beta_1\cdots\beta_n/n!$, its area is
\begin{equation}\label{eq:base-area}
  \cH^{n-1}(\conv\{b_1,\ldots,b_n\})
  =\frac{\beta_2\cdots\beta_n}{(n-1)!}.
\end{equation}

Let $\gamma_n$ denote the~standard Gaussian measure on $\R^n$.
By rotational invariance and polar coordinates, the~normalized solid
angle of every measurable cone $C\subset\R^n$ satisfies
$$
  \frac{\vol_n(C\cap\B^n)}{\kappa_n}
  =\gamma_n(C).
$$
Using $x_j=\beta_jz_j$ and the~notation of Lemma \ref{lem:ordered-square}, we find
\begin{equation}\label{eq:gaussian-angle}
  \gamma_n(C_{\beta})
  =\frac{\beta_1\cdots\beta_n}{(2\uppi)^{n/2}}
    \mathcal I(q_1,\ldots,q_n).
\end{equation}
Combining \eqref{eq:sector-equality}--\eqref{eq:gaussian-angle} yields
\begin{align}
  \frac{\cH^{n-1}(\conv\{b_1,\ldots,b_n\})}
       {\vol_n(B\cap\B^n)}
  &=\frac{(2\uppi)^{n/2}}
          {(n-1)!\,\kappa_n\,\beta_1\mathcal I(q_1,\ldots,q_n)}.\label{eq:ratio-I}
\end{align}
By Lemma \ref{lem:ordered-square} and $\beta_1=\sqrt{q_1}$, the~denominator in \eqref{eq:ratio-I} is no larger than the~corresponding denominator for $q_1=\ldots=q_n=1$.

Finally, up to their measure-zero boundaries, the~space $\R^n$ is divided into $2^nn!$ congruent regions by independently choosing the~signs and ordering the~absolute values of the~coordinates and $D_n$ is one of these regions.  Hence
$$
  \frac{\mathcal I(1,\ldots,1)}{(2\uppi)^{n/2}}=\frac1{2^nn!}.
$$
Substitution into \eqref{eq:ratio-I} gives
$$
  \frac{(2\uppi)^{n/2}}
       {(n-1)!\,\kappa_n \mathcal I(1,\ldots,1)}
  =\frac{2^nn!}{(n-1)!\,\kappa_n}
  =\frac{n2^n}{\kappa_n},
$$
which is \eqref{eq:orthoscheme-sector}.
\end{proof}

\section{Proof of the~main theorem and consequences}\label{sec:global}

\begin{proof}[Proof of Theorem \ref{thm:main}]
The~case $n=1$ is immediate, so suppose $n\geqslant2$.  By Lemma \ref{lem:generic}, we may first assume the~genericity condition stated there.  The~facet-distance part of \eqref{eq:face-distance} implies
\begin{equation}\label{eq:ball-contained}
  \B^n\subset P.
\end{equation}

For every nondegenerate flag simplex $A_{\mathcal F}$, combine Propositions \ref{prop:local-reduction} and \ref{prop:orthoscheme} to obtain
\begin{equation}\label{eq:flag-local-final}
  \cH^{n-1}(\conv\{a_1,\ldots,a_n\})
  \geqslant \frac{n2^n}{\kappa_n}
         \vol_n(A_{\mathcal F}\cap\B^n).
\end{equation}
If a~flag simplex is degenerate, its opposite facet is also degenerate: the~latter lies in $\aff F_1$, which does not contain the~origin.  Hence degenerate flags contribute zero to both measure-theoretic partitions in \eqref{eq:flag-partitions}.  Summing \eqref{eq:flag-local-final} over all maximal flags and using \eqref{eq:flag-partitions} and \eqref{eq:ball-contained}, we get
\begin{align*}
  \cH^{n-1}(\partial P)
  \geqslant\frac{n2^n}{\kappa_n}
       \sum_{\mathcal F}\vol_n(A_{\mathcal F}\cap\B^n)
  =\frac{n2^n}{\kappa_n}\vol_n(\B^n)
   =n2^n.
\end{align*}
The~generic-center reduction (Lemma \ref{lem:generic}) then gives the~result for the~original polytope.
\end{proof}

\begin{proof}[Proof of Corollary \ref{cor:cube}]
Identify $L$ isometrically with $\R^n$ and set $P=[-1,1]^N\cap L$.  This face-distance reduction is the~one used in \cite[Section~2]{Karasev}. We include the~details.  Let $F$ be a~codimension-$k$ face of $P$.  The~active coordinate constraints defining $F$ have rank $k$ when restricted to $L$.  Choose $k$ independent ones.  There are distinct indices $i_1,\ldots,i_k$ and signs $\varepsilon_j\in\{-1,1\}$ such that
$$
  F\subseteq
  \{x\in\R^N:x_{i_j}=\varepsilon_j,\ j=1,\ldots,k\}.
$$
The~affine subspace on the~right is at distance $\sqrt{k}$ from the~origin.  Since $\aff F$ is contained in it, $\dist(0,\aff F)\geqslant\sqrt{k}$.
Thus $P$, viewed intrinsically in $L$, satisfies the~hypotheses of Theorem \ref{thm:main}, which gives \eqref{eq:cube-section}.
\end{proof}

The~surface-area estimate also contains the~polyhedral volume estimate \cite[Theorem~1.1]{Karasev} from which Vaaler's theorem \cite{Vaaler} follows.

\begin{corollary}\label{cor:volume}
Under the~hypotheses of Theorem \ref{thm:main}, $\vol_n(P)\geqslant2^n$.
\end{corollary}

\begin{proof}
For a~facet $F$, let $h_F=\dist(0,\aff F)$.  Since $0\in\operatorname{int}P$, the~pyramids $\conv(\{0\}\cup F)$, indexed by the~facets of $P$, cover $P$ and have pairwise disjoint interiors.  Each has volume $h_F\cH^{n-1}(F)/n$, and therefore (see also the~general support-function formula in \cite[Chapter~5]{Schneider})
$$
  n\vol_n(P)=\sum_{F\text{ facet}}h_F\cH^{n-1}(F).
$$
Since $h_F\geqslant1$, Theorem \ref{thm:main} implies
$$
  n\vol_n(P)
  \geqslant\cH^{n-1}(\partial P)
  \geqslant n2^n.
$$
\end{proof}

\begin{remark}\label{rem:method}
The~only dimension-dependent object in the~proof is the~standard chamber $D_n$, whose Gaussian measure is fixed by the~$2^nn!$ symmetry count.  All geometric information about a~flag enters through the~prefix sums $r_k=|b_k|^2$.  This is why the~ordered-square substitution replaces the~separate circular and spherical-triangle arguments in dimensions two and three by one comparison valid for every $n$.
\end{remark}

\section{Higher skeleta and a~ridge conjecture}\label{sec:skeleta}

\noindent For $0\leqslant j\leqslant n$, let $\mathcal F_j(P)$ be the~set of $j$-dimensional faces of $P$, with $\mathcal F_n(P)=\{P\}$, and define the~total $j$-face measure by $\Phi_j(P)=\sum_{F\in\mathcal F_j(P)}\cH^j(F)$.  Thus $\Phi_n(P)=\vol_n(P)$ and $\Phi_{n-1}(P)=\cH^{n-1}(\partial P)$.  This is the~usual Hausdorff measure of the~$j$-skeleton of a~polytope. Skeletal measures and their integral-geometric behavior are studied, for example, in \cite{BurtonMeasure,BurtonSections}, while fixed-volume inequalities for sums of face measures in the~simplicial class appear in \cite{DallaTamvakis}.  For the~cube $Q_n=[-1,1]^n$, one has
\begin{equation}\label{eq:cube-skeleton}
  \Phi_{n-k}(Q_n)=2^n\binom{n}{k},
  \qquad 0\leqslant k\leqslant n.
\end{equation}
Indeed, $Q_n$ has $2^k\binom{n}{k}$ faces of codimension $k$, each of $(n-k)$-volume $2^{n-k}$.

It is tempting to read Theorem \ref{thm:main} as the~case $k=1$ of the~alleged family
\begin{equation}\label{eq:naive-skeleton}
  \Phi_{n-k}(P)\stackrel{?}{\geqslant}2^n\binom{n}{k}.
\end{equation}
The~next observation shows that this extension is false, including in the~original class of linear cube sections.

\begin{proposition}\label{prop:naive-false}
Inequality \eqref{eq:naive-skeleton} fails for polytopes satisfying all the~face-distance assumptions of Theorem \ref{thm:main}.  It also fails, in a~different dimension, for a~linear section of a~cube.
\end{proposition}

\begin{proof}
Let $S_4$ be a~regular $4$-simplex of inradius one, centered at the~origin, and denote its vertices by $v_0,\ldots,v_4$.  Since $\sum_i v_i=0$, the~centroid of the~facet opposite $v_i$ is $-v_i/4$. The~inradius assumption therefore gives $|v_i|=4$.  Regularity and the~same zero-sum identity then give $|v_i|^2=16$ and $v_i\mathbin{\cdot}v_j=-4$ for $i\neq j$.  A~face of codimension $r$ contains $5-r$ vertices. Hence symmetry shows that the~point of its affine hull nearest to the~origin is their centroid, and the~preceding inner products give
$$
  \dist(0,\aff F)^2=\frac{4r}{5-r}\geqslant r.
$$
Thus $S_4$ is admissible.  The~same inner products give edge length $|v_i-v_j|=2\sqrt{10}$, and hence
$$
  \Phi_1(S_4)=10\cdot2\sqrt{10}=20\sqrt{10}
  \approx63.2456<64=2^4\binom{4}{3}.
$$

For an~example which is itself a~cube section, put $L_0=\{x\in\R^4:x_1+x_2+x_3+x_4=0\}$ and $\mathsf O=[-1,1]^4\cap L_0$.  The~six vertices of $\mathsf O$ are the~permutations of $(1,1,-1,-1)$.  After an~orthogonal identification of $L_0$ with $\R^3$, the~section is $\conv\{\pm2\mathbf e_1,\pm2\mathbf e_2,\pm2\mathbf e_3\}$, where $\mathbf e_1,\mathbf e_2,\mathbf e_3$ denote the~standard orthonormal basis vectors of $\R^3$. It is therefore a~regular octahedron with $12$ edges of length $2\sqrt{2}$, eight triangular faces of area $2\sqrt{3}$, and volume $32/3$.  Define $\mathcal P_K(t)=\sum_{j=0}^{\dim K}\Phi_j(K)t^j$.  Direct face enumeration gives
\begin{equation}\label{eq:octahedron-face-polynomial}
  \mathcal P_{\mathsf O}(t)
  =6+24\sqrt{2}\,t+16\sqrt{3}\,t^2+\frac{32}{3}t^3.
\end{equation}
The~$j$-faces of a~Cartesian product $K\times M$ are precisely the~products $F\times G$, where $F\in\mathcal F_r(K)$, $G\in\mathcal F_s(M)$, and $r+s=j$ (see \cite[Proposition~2.6.13]{PinedaVillavicencio}).  Since the~affine hulls of $F$ and $G$ lie in orthogonal coordinate subspaces, Fubini's theorem gives
$$
\mathcal H^{r+s}(F\times G)
=\mathcal H^r(F)\mathcal H^s(G).
$$
Summing over all faces and collecting terms of equal degree therefore yields
$$
\mathcal P_{K\times M}(t)
=\mathcal P_K(t)\mathcal P_M(t).
$$
Consequently, $P:=\mathsf O^{\times3}=[-1,1]^{12}\cap L_0^{\times3}$ is a~$9$-dimensional linear cube section and satisfies
$$
  \Phi_1(P)=\left.\frac{\dd}{\dd t}\mathcal P_{\mathsf O}(t)^3\right|_{t=0}
  =3(24\sqrt{2})6^2=2592\sqrt{2}
  \approx3665.642<4608=2^9\binom{9}{8}.
$$
More generally, for every fixed $j$, the~constant term $6$ in \eqref{eq:octahedron-face-polynomial} yields
$$
  \frac{\Phi_j(\mathsf O^{\times m})}
       {2^{3m}\binom{3m}{j}}
  =\mathrm{O}_j\!\left(\left(\frac34\right)^m\right)
  \qquad (m\rightarrow\infty).
$$
Thus every fixed low-dimensional skeleton eventually violates the~cube value, even within linear sections.
\end{proof}

The~counterexamples above concern skeleta of large codimension.  The~first case beyond surface area, namely the~total measure of the~ridges, remains plausible.

\begin{conjecture}[Ridge-skeleton conjecture]\label{conj:ridge}
Let $n\geqslant3$, and let $P\subset\R^n$ satisfy the~hypotheses of Theorem \ref{thm:main}.  Then
\begin{equation}\label{eq:ridge-conjecture}
  \Phi_{n-2}(P)\geqslant2^n\binom{n}{2}
  =n(n-1)2^{n-1}.
\end{equation}
Equality is attained by $[-1,1]^n$.
\end{conjecture}

The~restriction $n\geqslant3$ is necessary: in dimension two, an~equilateral triangle of inradius one is admissible but has only three vertices, whereas the~right-hand side of \eqref{eq:ridge-conjecture} is four.  For $n=3$, Conjecture \ref{conj:ridge} is the~classical theorem of Besicovitch and Eggleston \cite{BesicovitchEggleston}: every convex polyhedron containing the~unit ball has total edge length at least $24$.  The~extra face-distance assumptions are essential to any higher-dimensional analogue based on the~cube constant.  For example, the~regular $24$-cell of inradius one has total $2$-face area $48\sqrt{3}<96$ \cite[p.~93]{FejesToth} -- it does not contradict Conjecture \ref{conj:ridge}, because its vertices have norm $\sqrt{2}<2$ and therefore violate the~codimension-$4$ condition. Related fixed-volume variational problems for the~same
codimension-two skeletal functional were studied by Scott
\cite{ScottCodimensionTwo}, who established the~existence and partial
regularity of a~minimizer in a~suitable generalized polyhedral class. 

Two exact classes support Conjecture \ref{conj:ridge}.

\begin{proposition}\label{prop:ridge-classes}
Conjecture \ref{conj:ridge} holds for admissible centered orthotopes.  It also holds for every centrally symmetric admissible polytope all of whose facets are at distance one from the~origin.
\end{proposition}

\begin{proof}
Let $R=\prod_{i=1}^n[-a_i,a_i]$.  Admissibility of the~facets gives $a_i\geqslant1$ for every $i$, and direct enumeration gives, more generally for every $1\leqslant k\leqslant n$,
$$
  \Phi_{n-k}(R)
  =2^n\sum_{\substack{J\subseteq\{1,\ldots,n\}\\|J|=n-k}}
       \prod_{j\in J}a_j
  \geqslant2^n\binom{n}{k}.
$$

Now let $P$ be centrally symmetric and suppose that every facet $F$ has distance one from the~origin.  Let $p_F$ be the~orthogonal projection of the~origin onto $\aff F$.  Since $\B^n\subset P$ and $\aff F$ is tangent to $\B^n$, the~point $p_F$ lies in $\relint F$: otherwise a~second facet through $p_F$ would give a~distinct supporting hyperplane of the~Euclidean ball at $p_F$.  If $G$ is a~face of codimension $r$ within $F$, then it has codimension $r+1$ within $P$, and orthogonal decomposition inside $\aff F$ gives
$$
  \dist_{\aff F}(p_F,\aff G)^2
  =\dist(0,\aff G)^2-\dist(0,\aff F)^2
  \geqslant(r+1)-1=r.
$$
Theorem \ref{thm:main}, applied intrinsically to the~$(n-1)$-polytope $F-p_F$, therefore gives
$$
  \cH^{n-2}(\partial F)\geqslant(n-1)2^{n-1}.
$$
A~full-dimensional centrally symmetric polytope has at least $2n$ facets: its facets occur in antipodal pairs, and their outer normals span $\R^n$.  Since every ridge belongs to exactly two facets, summing the~last inequality over all facets yields
$$
  2\Phi_{n-2}(P)
  =\sum_{F\in\mathcal F_{n-1}(P)}\cH^{n-2}(\partial F)
  \geqslant2n(n-1)2^{n-1},
$$
which is \eqref{eq:ridge-conjecture}.
\end{proof}

The~main theorem also gives a~general, although nonsharp, ridge estimate.

\begin{proposition}\label{prop:ridge-general}
Under the~hypotheses of Theorem \ref{thm:main}, for $n\geqslant3$, one has
\begin{equation}\label{eq:ridge-general}
  \Phi_{n-2}(P)
  \geqslant\frac{n-1}{2}(2\kappa_{n-1})^{1/(n-1)}
       (n2^n)^{(n-2)/(n-1)}\sim\frac{\sqrt{2\uppi\mathrm{e}}}{4}2^nn^{3/2}\approx 1.0332\cdot 2^nn^{3/2}.
\end{equation}
\end{proposition}

\begin{proof}
Write $S=\Phi_{n-1}(P)$, let $A_F=\cH^{n-1}(F)$ for every facet $F$, and put $q=(n-2)/(n-1)$.  The~Euclidean isoperimetric inequality in dimension $n-1$ \cite[Section~7.2]{Schneider} gives
$$
  \cH^{n-2}(\partial F)
  \geqslant(n-1)\kappa_{n-1}^{1/(n-1)}A_F^q.
$$
If $u_F$ denotes the~outer unit normal of a~facet $F$, then the~
divergence theorem applied to an~arbitrary constant vector field
$X(x)=v$ gives
\begin{align}
  \label{eq:divergencethm}
  0=\int_P\mathrm{div} X\dd x
  =\int_{\partial P}\langle v,u(x)\rangle\dd\cH^{n-1}(x)
  =\left\langle v,\sum_F A_Fu_F\right\rangle.
\end{align}
Since \eqref{eq:divergencethm} holds for every $v\in\R^n$, we obtain the~equilibrium
identity
$$
  \sum_F A_Fu_F=0.
$$
For a~fixed facet $F_0$, taking the~scalar product with $u_{F_0}$
and using $\langle u_F,u_{F_0}\rangle\geqslant-1$ (the~vectors $u_F$ and $u_{F_0}$ are unit) gives
$$
  0
  =A_{F_0}
   +\sum_{F\neq F_0}A_F\langle u_F,u_{F_0}\rangle
  \geqslant A_{F_0}-\sum_{F\neq F_0}A_F
  =2A_{F_0}-S.
$$
Consequently, $A_{F_0}\leqslant S/2$, and hence $A_F\leqslant S/2$
for every facet $F$.  Since $q-1<0$, it follows that
$$
  \sum_F A_F^q
  =\sum_F A_FA_F^{q-1}
  \geqslant\left(\frac S2\right)^{q-1}\sum_F A_F
  =2^{1/(n-1)}S^q.
$$
Every ridge is counted twice among the~boundaries of the~facets.  Therefore
$$
  2\Phi_{n-2}(P)
  \geqslant(n-1)(2\kappa_{n-1})^{1/(n-1)}S^{(n-2)/(n-1)}.
$$
Theorem \ref{thm:main} gives $S\geqslant n2^n$, and the inequality in \eqref{eq:ridge-general} follows. By Stirling's formula, the~right-hand side of the inequality in  \eqref{eq:ridge-general} is asymptotic to $\frac{\sqrt{2\uppi\mathrm e}}{4}\,2^n n^{3/2}$ as $n\to\infty$.
\end{proof}

Finally, the~ordered-square comparison does not extend to ridges by simply inserting one more orthoscheme edge length.  Such a~pointwise extension would lead to the~inequality $\sqrt{q_1q_2}\,I(q_1,q_2)\leqslant I(1,1)$.  In dimension two, however, the~substitution $x_i=\sqrt{q_i}z_i$ maps $D_2$ to a~Gaussian wedge of angle $\arctan\sqrt{q_2/q_1}$, and polar coordinates give
$$
  \sqrt{q_1q_2}\,I(q_1,q_2)
  =\arctan\sqrt{\frac{q_2}{q_1}}.
$$
The~admissible choice $(q_1,q_2)=(1,M)$ makes the~right-hand side tend to $\uppi/2$, whereas $I(1,1)=\uppi/4$.  Thus any proof of Conjecture \ref{conj:ridge} must use information lost by a~single-flag comparison.  

\begin{question}[Range of valid codimensions]
\label{question:valid-codimensions}
For which pairs $(n,k)$, with $1\leq k\leq n$, does the~inequality
$$
  \Phi_{n-k}(P)\geq 2^n\binom{n}{k}
$$
hold for every polytope $P\subset\R^n$ satisfying the~hypotheses of Theorem~\ref{thm:main}?  In particular, determine the~largest integer $k_{\max}(n)$ for which it holds simultaneously for all $1\leq k\leq k_{\max}(n)$.
\end{question}

\section*{Declarations}

\noindent\textbf{Funding.}
The~author received no financial support for the~research, authorship,
or publication of this article.

\noindent\textbf{Conflict of interest.}
The~author declares that there are no competing interests.

\noindent\textbf{Data availability.}
No data were generated or analyzed in this theoretical study.


\begin{thebibliography}{99}

\bibitem{BallProdromou}
K.~M.~Ball and M.~Prodromou,
\emph{A~sharp combinatorial version of Vaaler's theorem},
{Bull. Lond. Math. Soc.} \textbf{41} (2009), no.~5, 853--858,
\href{https://doi.org/10.1112/blms/bdp062}
{doi:10.1112/blms/bdp062}.

\bibitem{BesicovitchEggleston}
A.~S.~Besicovitch and H.~G.~Eggleston, \emph{The~total length of the~edges of a~polyhedron}, {Quart. J. Math.,} Volume \textbf{8}, Issue 1 (1957), 172--190, \href{https://doi.org/10.1093/qmath/8.1.172}{doi:10.1093/qmath/8.1.172}.

\bibitem{BurtonMeasure}
G.~R.~Burton, \emph{The~measure of the~$s$-skeleton of a~convex body}, {Mathematika} \textbf{26} (1979), no.~2, 290--301, \href{https://doi.org/10.1112/S0025579300009839}{doi:10.1112/S0025579300009839}.

\bibitem{BurtonSections}
G.~R.~Burton, \emph{Skeleta and sections of convex bodies}, {Mathematika} \textbf{27} (1980), no.~1, 97--103, \href{https://doi.org/10.1112/S0025579300009980}{doi:10.1112/S0025579300009980}.

\bibitem{DallaTamvakis}
L.~Dalla and N.~K.~Tamvakis, \emph{An~isoperimetric inequality in the~class of simplicial polytopes}, {Math. Japon.} \textbf{44} (1996), no.~3, 569--572.

\bibitem{FejesToth}
L.~Fejes T\'oth, \emph{On the~total area of the~faces of a~four-dimensional polytope}, {Canad. J. Math.} \textbf{17} (1965), 93--99, \href{https://doi.org/10.4153/CJM-1965-009-x}{doi:10.4153/CJM-1965-009-x}.

\bibitem{GodlandKabluchko}
T.~Godland and Z.~Kabluchko,
\emph{Conic intrinsic volumes of Weyl chambers},
{Modern Stochastics: Theory and Applications}
\textbf{9} (2022), no.~3, 357--375,
\href{https://doi.org/10.15559/22-VMSTA206}
{doi:10.15559/22-VMSTA206}.

\bibitem{Karasev}
R.~Karasev, \emph{Rogers's proof of Vaaler's theorem}, {Discrete Math.} \textbf{349} (2026), no.~8, 115088, \href{https://doi.org/10.1016/j.disc.2026.115088}{doi:10.1016/j.disc.2026.115088}.

\bibitem{PinedaVillavicencio}
G.~Pineda Villavicencio, \emph{Polytopes and Graphs}, Cambridge Studies in Advanced Mathematics, vol.~211, Cambridge University Press, Cambridge, 2024, \href{https://doi.org/10.1017/9781009257794}{doi:10.1017/9781009257794}.


\bibitem{Rogers}
C.~A.~Rogers, \emph{The~packing of equal spheres}, {Proc. London Math. Soc.} (3) \textbf{8} (1958), 609--620, \href{https://doi.org/10.1112/plms/s3-8.4.609}{doi:10.1112/plms/s3-8.4.609}.

\bibitem{Schneider}
R.~Schneider, \emph{Convex Bodies: The~Brunn--Minkowski Theory}, 2nd expanded ed., Encyclopedia of Mathematics and its Applications, vol.~151, Cambridge University Press, Cambridge, 2014, \href{https://doi.org/10.1017/CBO9781139003858}{doi:10.1017/CBO9781139003858}.

\bibitem{ScottCodimensionTwo}
R.~C.~Scott,
\emph{Minimizing the~mass of the~codimension two skeleton of convex,
unit volume polyhedra},
{Indiana Univ. Math. J.} \textbf{61} (2012), no.~4, 1513--1564,
\href{https://doi.org/10.1512/iumj.2012.61.4734}
{doi:10.1512/iumj.2012.61.4734}.

\bibitem{Vaaler}
J.~D.~Vaaler, \emph{A~geometric inequality with applications to linear forms}, {Pacific J. Math.} \textbf{83} (1979), no.~2, 543--553, \href{https://doi.org/10.2140/pjm.1979.83.543}{doi:10.2140/pjm.1979.83.543}.

\bibitem{Zong}
C.~Zong, \emph{The~Cube: A~Window to Convex and Discrete Geometry}, Cambridge University Press, 2006, \href{https://doi.org/10.1017/CBO9780511543173}{doi:10.1017/CBO9780511543173}.

\end{thebibliography}
\end{document}